\documentclass [a4paper,11pt, reqno]{amsart}
\usepackage{amsfonts}
\usepackage{amssymb}
\usepackage{latexsym}
\usepackage[mathscr]{eucal}
\usepackage[T1]{fontenc}
\usepackage[utf8]{inputenc}
\usepackage{latexsym}
\usepackage{amsmath}
\usepackage{amsthm}
\usepackage{mathtools}

\allowdisplaybreaks

\numberwithin{equation}{section}

\begin{document}

\title[Inequalities of the Edmundson-Lah-Ribari\v c type]{Inequalities of the Edmundson-Lah-Ribari\v c type for $n$-convex functions with applications}
\author[R. Miki\' c]{Rozarija Miki\' c}

\address{Faculty of Textile Technology, University of Zagreb\\
Prilaz baruna Filipovi\' ca 28a\\
10 000 Zagreb\\
Croatia}

\email{rozarija.jaksic@ttf.hr}

\author[\dj . Pe\v cari\' c]{\dj ilda Pe\v cari\' c}
\address{Catholic University of Croatia\\
Ilica 242\\
10 000 Zagreb\\
Croatia}
\email{gildapeca@gmail.com}

\author[J. Pe\v cari\' c]{Josip Pe\v cari\' c}
\address{RUDN University\\ Miklukho-Maklaya str. 6\\ 117198 Moscow\\ Russia}
\email{pecaric@element.hr}

\subjclass[2010]{26A51, 15A39, 60E15}

\keywords{Jensen inequality, Edmundson-Lah-Ribari\v c inequality, $n$-convex functions, divided differences, $f$-divergence, Zipf-Mandelbrot law}

\begin{abstract}
In this paper we derive some Edmundson-Lah-Ribari\v c type inequalities for positive linear functionals and $n$-convex functions. Main results are applied to the generalized $f$-divergence functional. Examples with Zipf–Mandelbrot law are used to illustrate the results.
\end{abstract}

\maketitle

\section{Introduction}

\newtheorem{tm}{Theorem}[section]
\newtheorem{kor}{Corollary}[section]
\newtheorem{lema}{Lemma}[section]
\newtheorem{rem}{Remark}[section]
\newtheorem{primjer}{Example}[section]
\newtheorem{df}{Definition}[section]

Let $E$ be a non-empty set and let $L$ be a vector space of real-valued functions $f \colon E \to \mathbb{R}$
having the properties:
\begin{itemize}
\item[(L1)]
 $f,g \in L \Rightarrow (af+bg)\in L$ for all $a,b \in\mathbb{R}$;
\item[(L2)]
 $\boldsymbol{1} \in L$, i.e., if $f(t)=1$ for every $t \in E$, then $f \in L$. 
\end{itemize}

We also consider positive linear functionals $A \colon L \to \mathbb{R}$. That is, we assume that:
\begin{itemize}
\item[(A1)]
$A(af+bg)=aA(f)+bA(g)$ for $f,g \in L$ and $a,b \in \mathbb{R}$;
\item[(A2)]
$f \in L$, $f(t) \ge 0$ for every $t \in E \Rightarrow A(f) \ge 0$ ($A$ is positive). 
\end{itemize}

Since it was proved, the famous Jensen inequality and its converses have been extensively studied by many authors and have been generalized in numerous directions. Jessen \cite{jessen} gave the following generalization of Jensen's inequality for convex functions (see also \cite[p.47]{crvena}):
\begin{tm} $\mathrm{(}$\cite{jessen}$\mathrm{)}$
Let $L$ satisfy properties (L1) and (L2) on a nonempty set $E$, and assume that $f$ is a continous convex function 
on an interval $I \subset \mathbb{R}$. If $A$ is a positive linear functional with $A(1)=1$, then for all $g \in L$ such
that $f(g) \in L$ we have $A(g) \in I$ and
\begin{equation}\label{jensen}
f(A(g)) \le A(f (g)).
\end{equation} 
\end{tm}

The following result is one of the most famous converses of the Jensen inequality known as the Edmundson-Lah-Ribari\v c inequality, and it was proved in \cite{beesack} by Beesack and Pe\v{c}ari\'{c} (see also \cite[p.98]{crvena}):
\begin{tm} $\mathrm{(}$\cite{beesack}$\mathrm{)}$
Let $f$ be convex on the interval $I=[a,b]$ such that $-\infty <a<b<\infty$. Let $L$ satisfy conditions (L1) and (L2) on $E$ and let $A$ be any
 positive linear functional on $L$ with $A(1)=1$. Then for every $g \in L$ such that $f(g) \in L$ (so that
$a \le g(t) \le b$ for all $t \in E$), we have
\begin{equation} \label{1}
A(f(g)) \le \frac{b-A(g)}{b-a}f(a)+\frac{A(g)-a}{b-a}f(b).
\end{equation} 
\end{tm}

For some recent results on the converses of the Jensen inequality, the reader is referred to \cite{hor}, \cite{treci}, \cite{mon6}, \cite{hor1}, \cite{ivelic} and \cite{moh}.

Unlike the results from the above mentioned papers, which require convexity of the involved functions, the main objective of this paper is to obtain inequalities of the Edmundson-Lah-Ribari\v c type that hold for $n$-convex functions, which will also be a generalization of the results from \cite{prvi3conv} and \cite{prvi3conv_eksp}.

Definition of $n$-convex functions is characterized by $n$-th order divided differences.
The $n$-th order divided difference of a function $f\colon [a, b] \to \mathbb{R}$ at mutually distinct points $t_0, t_1, ... , t_n\in  [a, b]$ is defined recursively by
\begin{align*}
[t_i]f&=f(t_i), \ \ i=0,...,n,\\
[t_0,...,t_n]f&=\dfrac{[t_1,..., t_n]f-[t_0,..., t_{n-1}]f}{t_n-t_0}.
\end{align*}
The value $[t_0, . . . , t_n]f$ is independent of the order of the points $t_0, . . . , t_n$.

Definition of divided differences can be extended to include the cases in which some or all the points coincide (see e.g. \cite{agarwal}, \cite{crvena}):
\begin{equation*}
    f[\underbrace{a,...,a}_{n \ times}]=\dfrac{1}{(n-1)!}f^{(n-1)}(a), \ \ n\in \mathbb{N}.
\end{equation*}

A function $f\colon [a, b] \to \mathbb{R}$ is said to be $n$-convex ($n \ge 0$)
if and only if for all choices of $(n+1)$ distinct points $t_0, t_1, ... , t_n\in  [a, b]$, we have $[t_0, . . . , t_n]f\ge 0$.

The results in this paper are obtained by utilizing Hermite's interpolating polynomial, so first we need to give a definition and some properties (see \cite{agarwal}).

Let $-\infty <a<b<\infty$ and let $a\le a_1<a_2<...<a_r\le b$, where $r\ge 2$, be given points. For $f\in \mathcal{C}^n([a,b])$ there exists a unique polynomial $P_H(t)$, called Hermite's interpolating polynomial, of degree $(n-1)$ fulfilling Hermite's conditions:
$$P_H^{(i)}(a_j)=f^{(i)}(a_j): \ 0\le i\le k_j, \ 1\le j\le r, \ \sum_{j=1}^rk_j+r=n.$$
Among other special cases, these conditions include type $(m,n-m)$ conditions, which will be of special interest to us:
$$(r=2, \ 1\le m\le n-1, \ k_1=m-1, \ k_2=n-m-1)$$
\begin{align*}
    P_{mn}^{(i)}(a)&=f^{(i)}(a), \ 0\le i\le m-1\\
    P_{mn}^{(i)}(b)&=f^{(i)}(b), \ 0\le i\le n-m-1.
\end{align*}

To give a development of the interpolating polynomial in terms of divided differences, first let us assume that the function $f$ is also defined at a point $t\neq a_j, \ 1\le j\le n$. In \cite{agarwal} it is shown that
\begin{align}
    f(t)&=P(t)+R(t),\label{prikaz}
\end{align}
where 
\begin{align}
    P(t)=&f(a_1)+(t-a_1)f[a_1,a_2]+(t-a_1)(t-a_2)f[a_1,a_2,a_3]\nonumber\\
    &+...+(t-a_1)\cdot \cdot \cdot (t-a_{n-1})f[a_1,...,a_n]\label{polinom}
\end{align}
and
\begin{align}
    R(t)=&(t-a_1)\cdot \cdot \cdot (t-a_{n})f[t,a_1,...,a_n].\label{ostatak}
\end{align}

In case of $(m,n-m)$ conditions, (\ref{polinom}) and (\ref{ostatak}) become
\begin{align}
    P_{mn}(t)=&f(a)+(t-a)f[a,a]+...+(t-a)^{m-1}f[\underbrace{a,...,a}_{m \ times}]\nonumber\\
    &+(t-a)^mf[\underbrace{a,...,a}_{m \ times};b]+(t-a)^{m}(t-b)f[\underbrace{a,...,a}_{m \ times};b,b]\nonumber\\
    &+...+(t-a)^{m}(t-b)^{n-m-1}f[\underbrace{a,...,a}_{m \ times};\underbrace{b,b,...,b}_{(n-m) \ times}]\label{polinom1}
\end{align}
and
\begin{align}
    R_m(t)=&(t-a)^{m}(t-b)^{n-m}f[t;\underbrace{a,...,a}_{m \ times};\underbrace{b,b,...,b}_{(n-m) \ times}].\label{ostatak1}
\end{align}

This paper is organized in the following manner: main results, that are inequalities of the Edmundson-Lah-Ribari\v c type for $n$-convex functions, are given in Section \ref{sec2}; application of the main results to the generalized $f$-divergence functional is given in Section \ref{sec3}, and finally in section \ref{sec4} the results for the generalized $f$-divergence are applied to Zipf–Mandelbrot law.

\section{Results}\label{sec2}

Throughout this paper, whenever mentioning the interval $[a,b]$, we assume that $-\infty<a<b<\infty$ holds.

Let $L$ satisfy conditions $(L1)$ and $(L2)$ on a non-empty set $E$, let $A$ be any  positive linear functional on $L$ with $A(\boldsymbol{1})=1$, and let $g\in L$ be any function such that $g(E)\subseteq [a,b]$. For a given function $f\colon [a,b]\to \mathbb{R}$ denote:
\begin{align}
    LR(f,g,a,b,A)&=A(f(g))-\frac{b-A(g)}{b-a}f(a)-\frac{A(g)-a}{b-a}f(b).\label{LRoznaka}
\end{align}

Following representations of the left side in the Edmundson-Lah-Ribari\v c inequality is obtained by using Hermite's interpolating polynomials in terms of divided differences (\ref{polinom1}).

\begin{lema}\label{lema2}
Let $L$ satisfy conditions (L1) and (L2) on a non-empty set $E$ and let $A$ be any  positive linear functional on $L$ with $A(\boldsymbol{1})=1$. Let $f\in \mathcal{C}^n([a,b])$, and let $g\in L$ be any function such that $f\circ g\in L$. Then the following identities hold:
\begin{align}
    \bullet \ LR(f,&g,a,b,A)=
    \sum_{k=2}^{n-1}f[a;\underbrace{b,...,b}_{k \ times}]A\left[(g-a\boldsymbol{1})(g-b\boldsymbol{1})^{k-1}\right]+A(R_1(g)) \label{lema2_rez1}\\
    \bullet \ LR(f,&g,a,b,A)=f[a,a;b]A[(g-a\boldsymbol{1})(g-b\boldsymbol{1})]+\nonumber\\
&\sum_{k=2}^{n-2}f[a,a;\underbrace{b,...,b}_{k \ times}]A\left[(g-a\boldsymbol{1})^2(g-b\boldsymbol{1})^{k-1}\right]+A(R_2(g)) \label{lema2_rez2} \\
    \bullet \ LR(f,&g,a,b,A)=(A(g)-a)\left(f[a,a]-f[a,b]\right)+\sum_{k=2}^{m-1}\dfrac{f^{(k)}(a)}{k!}A\left[(g-a\boldsymbol{1})^{k}\right]\nonumber\\
    &+\sum_{k=1}^{n-m}f[\underbrace{a,...,a}_{m \ times};\underbrace{b,...,b}_{k \ times}]A\left[(g-a\boldsymbol{1})^{m}(g-b\boldsymbol{1})^{k-1}\right]+A(R_m(g)), \label{lema2_rez3}
\end{align}
where $m\ge 3$ and $R_m(\cdot )$ is defined in (\ref{ostatak1}).
\end{lema}

\begin{proof} From representation (\ref{prikaz}) of every function $f\in \mathcal{C}^n([a,b])$ and its Hermite interpolating polynomial of type $(m,n-m)$ conditions in terms of divided differences (\ref{polinom1}) we have
\begin{align}
    f(t)&=f(a)+(t-a)f[a,a]+...+(t-a)^{m-1}f[\underbrace{a,...,a}_{m \ times}]\nonumber\\
    &+(t-a)^mf[\underbrace{a,...,a}_{m \ times};b]+(t-a)^{m}(t-b)f[\underbrace{a,...,a}_{m \ times};b,b]\nonumber\\
    &+...+(t-a)^{m}(t-b)^{n-m-1}f[\underbrace{a,...,a}_{m \ times};\underbrace{b,b,...,b}_{(n-m) \ times}]+R_m(t),\label{lema2_pom1}
\end{align}
where $R_m(\cdot )$ is defined in (\ref{ostatak1}). After some straightforward calculations, for different choices of $1\le m\le n-1$, from (\ref{lema2_pom1}) we get the following:
\begin{itemize}
    \item for $m=1$ it holds
\end{itemize}
\begin{align}
    LR(f,\boldsymbol{1},a,b,\mathrm{id})=&(t-a)(t-b)f[a;b,b]+(t-a)(t-b)^2f[a;b,b,b]\nonumber \\
    &+...+(t-a)(t-b)^{n-2}f[a;\underbrace{b,b,...,b}_{(n-1) \ times}]+R_1(t) \label{lema2_pom2a}
\end{align}
\begin{itemize}
    \item for $m=2$ it holds
\end{itemize}
\begin{align}
    LR(f,\boldsymbol{1},a,b,\mathrm{id})=&(t-a)(t-b)f[a,a;b]+(t-a)^2(t-b)f[a,a;b,b]\nonumber \\
&+...+(t-a)^2(t-b)^{n-3}f[a,a;\underbrace{b,b,...,b}_{(n-2) \ times}]+R_2(t) \label{lema2_pom2b} 
\end{align}
\begin{itemize}
    \item for $3\le m\le n-1$ it holds
\end{itemize}
\begin{align}
    LR(f,\boldsymbol{1},&a,b,\mathrm{id})=(t-a)\left(f[a,a]-f[a,b]\right)+...+(t-a)^{m-1}f[\underbrace{a,...,a}_{m \ times}]\nonumber\\
    &+(t-a)^mf[\underbrace{a,...,a}_{m \ times};b]+(t-a)^{m}(t-b)f[\underbrace{a,...,a}_{m \ times};b,b]\nonumber\\
    &+...+(t-a)^{m}(t-b)^{n-m-1}f[\underbrace{a,...,a}_{m \ times};\underbrace{b,b,...,b}_{(n-m) \ times}]+R_m(t).\label{lema2_pom2c}
\end{align}
Since $f\circ g\in L$ it holds $g(E)\subseteq [a,b]$, so we can replace $t$ with $g(t)$ in (\ref{lema2_pom2a}), (\ref{lema2_pom2b}) and (\ref{lema2_pom2c}), and thus obtain:
\begin{align*}
    LR(f,g,a,b,\mathrm{id})=&
    \sum_{k=2}^{n-1}(g(t)-a)(g(t)-b)^{k-1}f[a;\underbrace{b,...,b}_{k \ times}]+R_1(g(t)),
\end{align*}
\begin{align*}
    LR(f,g,a,&b,\mathrm{id})=(g(t)-a)(g(t)-b)f[a,a;b]+\\
&+\sum_{k=2}^{n-2}(g(t)-a)^2(g(t)-b)^{k-1}f[a,a;\underbrace{b,...,b}_{k \ times}]+R_2(g(t))
\end{align*}
and
\begin{align*}
    LR(f,g,&a,b,\mathrm{id})=(g(t)-a)\left(f[a,a]-f[a,b]\right)+\sum_{k=3}^m(g(t)-a)^{k-1}f[\underbrace{a,...,a}_{k \ times}]\nonumber\\
    &+\sum_{k=1}^{n-m}(g(t)-a)^{m}(g(t)-b)^{k-1}f[\underbrace{a,...,a}_{m \ times};\underbrace{b,...,b}_{k \ times}]+R_m(g(t)).
\end{align*}
Identities (\ref{lema2_rez1}), (\ref{lema2_rez2}) and (\ref{lema2_rez3}) follow by applying positive normalized linear functional $A$ to the previous equalities respectively. 
\end{proof}

\begin{lema}\label{lema3}
Let $L$ satisfy conditions (L1) and (L2) on a non-empty set $E$ and let $A$ be any  positive linear functional on $L$ with $A(\boldsymbol{1})=1$. Let $f\in \mathcal{C}^n([a,b])$, and let $g\in L$ be any function such that $f\circ g\in L$. Then the following identities hold:
\begin{align}
  \bullet \   LR(f,g,&a,b,A)=
    \sum_{k=2}^{n-1}f[b;\underbrace{a,...,a}_{k \ times}]A[(g-b\boldsymbol{1})(g-a\boldsymbol{1})^{k-1}]+A(R^{\ast}_1(g)) \label{lema3_rez1}\\
  \bullet \     LR(f,g,&a,b,A)=f[b,b;a]A[(g-b\boldsymbol{1})(g-a\boldsymbol{1})]\nonumber\\
&+\sum_{k=2}^{n-2}f[b,b;\underbrace{a,...,a}_{k \ times}]A[(g-b\boldsymbol{1})^2(g-a\boldsymbol{1})^{k-1}]+A(R^{\ast}_2(g)) \label{lema3_rez2}\\
 \bullet \  LR(f,g,&a,b,A)=(b-A(g))\left(f[a,b]-f[b,b]\right)+\sum_{k=2}^{m-1}\dfrac{f^{(k)}(b)}{k!}A[(g-b\boldsymbol{1})^{k}]\nonumber\\
    &+\sum_{k=1}^{n-m}f[\underbrace{b,...,b}_{m \ times};\underbrace{a,...,a}_{k \ times}]A[(g-b\boldsymbol{1})^{m}(g-a\boldsymbol{1})^{k-1}]+A(R^{\ast}_m(g))\label{lema3_rez3}
\end{align}
where $m\ge 3$ and
\begin{align}
    A(R^{\ast}_m(g))&=A[f[g;\underbrace{b\boldsymbol{1},...,b\boldsymbol{1}}_{m \ times};\underbrace{a\boldsymbol{1},...,a\boldsymbol{1}}_{(n-m) \ times}](g-b\boldsymbol{1})^{m}(g-a\boldsymbol{1})^{n-m}].\label{ostatak2}
\end{align}
\end{lema}

\begin{proof}
Let us define an auxiliary function $F\colon [a,b]\to \mathbb{R}$ with
$$F(t)=f(a+b-t).$$
Since $f\in \mathcal{C}^n([a,b])$ we immediately have $F\in \mathcal{C}^n([a,b])$, so we can apply (\ref{lema2_pom2a}), (\ref{lema2_pom2b}) and (\ref{lema2_pom2c}) to $F$ and obtain respectively
\begin{align}
    LR(F,\boldsymbol{1},&a,b,\mathrm{id})=\sum_{k=2}^{n-1}F[a;\underbrace{b,...,b}_{k \ times}](t-a)(t-b)^{k-1}+R_1(t) \label{lema3_pom1a}\\
    LR(F,\boldsymbol{1},&a,b,\mathrm{id})=F[a,a;b](t-a)(t-b)\nonumber\\
&+\sum_{k=2}^{n-2}F[a,a;\underbrace{b,...,b}_{k \ times}](t-a)^2(t-b)^{k-1}+R_2(t) \label{lema3_pom1b}\\
    LR(F,\boldsymbol{1},&a,b,\mathrm{id})=(t-a)\left(F[a,a]-F[a,b]\right)+\sum_{k=2}^{m-1}\dfrac{F^{(k)}(a)}{k!}(t-a)^{k}\nonumber\\
    &+\sum_{k=1}^{n-m}F[\underbrace{a,...,a}_{m \ times};\underbrace{b,...,b}_{k \ times}](t-a)^{m}(t-b)^{k-1}+R_m(t).\label{lema3_pom1c}
\end{align}
We can calculate divided differences of the function $F$ in terms of divided differences of the function $f$:
\begin{align*}
F[\underbrace{a,...,a}_{k \ times};\underbrace{b,...,b}_{i \ times}]&=(-1)^{k+i-1}f[\underbrace{b,...,b}_{k \ times};\underbrace{a,...,a}_{i \ times}].
\end{align*} 
Now (\ref{lema3_pom1a}), (\ref{lema3_pom1b}) and (\ref{lema3_pom1c}) become
\begin{align}
    LR(F,\boldsymbol{1},&a,b,\mathrm{id})=\sum_{k=2}^{n-1}(-1)^{k}f[b;\underbrace{a,...,a}_{k \ times}](t-a)(t-b)^{k-1}+\bar{R}_1(t) \label{lema3_pom2a}\\
    LR(F,\boldsymbol{1},&a,b,\mathrm{id})=(-1)^2f[b,b;a](t-a)(t-b)\nonumber\\
&+\sum_{k=2}^{n-2}(-1)^{k+1}f[b,b;\underbrace{a,...,a}_{k \ times}](t-a)^2(t-b)^{k-1}+\bar{R}_2(t) \label{lema3_pom2b}\\
    LR(F,\boldsymbol{1},&a,b,\mathrm{id})=(t-a)\left(-f[b,b]+f[a,b]\right)+\sum_{k=2}^{m-1}\dfrac{(-1)^kf^{(k)}(b)}{k!}(t-a)^{k}\nonumber\\
    &+\sum_{k=1}^{n-m}(-1)^{m+k-1}f[\underbrace{b,...,b}_{m \ times};\underbrace{a,...,a}_{k \ times}](t-a)^{m}(t-b)^{k-1}+\bar{R}_m(t),\label{lema3_pom2c}
\end{align}
where
\begin{align*}
    \bar{R}_m(t)&=(t-a)^{m}(t-b)^{n-m}(-1)^nf[a+b-t;\underbrace{b,...,b}_{m \ times};\underbrace{a,a,...,a}_{(n-m) \ times}].
\end{align*}
Let $g\in L$ be any function such that $f\circ g\in L$, that is, $a\le g(t)\le b$ for every $t\in E$. Let us define a function $\bar{g}(t)=a+b-g(t)$. Trivially, we have $a\le \bar{g}(t)\le b$ and $\bar{g}\in L$. Since
\begin{align*}
  LR(F,\bar{g},&a,b,\mathrm{id})=f(a+b-(a+b-g(t)))-\dfrac{b-(a+b-g(t))}{b-a}f(a+b-a)\\
  &-\dfrac{a+b-g(t)-a}{b-a}f(a+b-b)=LR(f,g,a,b,\mathrm{id}),  
\end{align*}
after putting $\bar{g}(t)$ in (\ref{lema3_pom2a}), (\ref{lema3_pom2b}) and (\ref{lema3_pom2c}) instead of $t$, we get
\begin{align*}
    LR(f,&g,a,b,\mathrm{id})=\sum_{k=2}^{n-1}(-1)^{k}f[b;\underbrace{a,...,a}_{k \ times}](b-g(t))(a-g(t))^{k-1}+\bar{R}_1(a+b-g(t))\\
    LR(f,&g,a,b,\mathrm{id})=(-1)^2f[b,b;a](b-g(t))(a-g(t))\nonumber\\
&+\sum_{k=2}^{n-2}(-1)^{k+1}f[b,b;\underbrace{a,...,a}_{k \ times}](b-g(t))^2(a-g(t))^{k-1}+\bar{R}_2(a+b-g(t))\\
    LR(f,&g,a,b,\mathrm{id})=(b-g(t))\left(-f[b,b]+f[a,b]\right)+\sum_{k=2}^{m-1}\dfrac{(-1)^kf^{(k)}(b)}{k!}(b-g(t))^{k}\nonumber\\
    &+\sum_{k=1}^{n-m}(-1)^{m+k-1}f[\underbrace{b,...,b}_{m \ times};\underbrace{a,...,a}_{k \ times}](b-g(t))^{m}(a-g(t))^{k-1}+\bar{R}_m(a+b-g(t)).
\end{align*}
Identities (\ref{lema3_rez1}), (\ref{lema3_rez2}) and (\ref{lema3_rez3}) follow after applying a normalized positive linear functional $A$ to previous equalities respectively. 
\end{proof}

Our first result is an upper bound for the difference in the Edmundson-Lah-Ribari\v c inequality, expressed by Hermite's interpolating polynomials in terms of divided differences.

\begin{tm}\label{tm2}
Let $L$ satisfy conditions (L1) and (L2) on a non-empty set $E$ and let $A$ be any  positive linear functional on $L$ with $A(\boldsymbol{1})=1$. Let $f\in \mathcal{C}^n([a,b])$, and let $g\in L$ be any function such that $f\circ g\in L$. If the function $f$ is $n$-convex and if $n$ and $m\ge 3$ are of different parity, then 
\begin{align}
    LR(f,g,a,&b,A)\le (A(g)-a)\left(f[a,a]-f[a,b]\right)+\sum_{k=2}^{m-1}\dfrac{f^{(k)}(a)}{k!}A\left[(g-a\boldsymbol{1})^{k}\right]\nonumber\\
    &+\sum_{k=1}^{n-m}f[\underbrace{a,...,a}_{m \ times};\underbrace{b,...,b}_{k \ times}]A\left[(g-a\boldsymbol{1})^{m}(g-b\boldsymbol{1})^{k-1}\right]. \label{tm2_rez1}
\end{align}
Inequality (\ref{tm2_rez1}) also holds when the function $f$ is $n$-concave and $n$ and $m$ are of equal parity. In case when the function $f$ is $n$-convex and $n$ and $m$ are of equal parity, or when the function $f$ is $n$-concave and $n$ and $m$ are of different parity, the inequality sign in (\ref{tm2_rez1}) is reversed.
\end{tm}

\begin{proof} We start with the representation of the left side in the Edmundson-Lah-Ribari\v c inequality (\ref{lema2_rez3}) from Lemma \ref{lema2} with a special focus on the last term:
\begin{align*}
    A(R(g))&=A\left(\left(g-a\boldsymbol{1}\right)^{m}\left(g-b\boldsymbol{1}\right)^{n-m}f[g;\underbrace{a\boldsymbol{1},...,a\boldsymbol{1}}_{m \ times};\underbrace{b\boldsymbol{1},...,b\boldsymbol{1}}_{(n-m) \ times}]\right).
\end{align*}
Since $A$ is positive, it preserves the sign, so we need to study the sign of the expression:
$$\left(g(t)-a\right)^{m}\left(g(t)-b\right)^{n-m}f[g(t);\underbrace{a,...,a}_{m \ times};\underbrace{b,b,...,b}_{(n-m) \ times}].$$

Since $a\le g(t)\le b$ for every $t\in E$, we have $\left(g(t)-a\right)^{m}\ge 0$ for every $t\in E$ and any choice of $m$. For the same reason we have $(g(t)-b)\le 0$. Trivially it follows that $(g(t)-b)^{n-m}\le 0$ when $n$ and $m$ are of different parity, and $(g(t)-b)^{n-m}\ge 0$ when $n$ and $m$ are of equal parity.

If the function $f$ is $n$-convex, then $f[g(t);\underbrace{a,...,a}_{m \ times};\underbrace{b,b,...,b}_{(n-m) \ times}]\ge 0$, and if the function $f$ is $n$-concave, then $f[g(t);\underbrace{a,...,a}_{m \ times};\underbrace{b,b,...,b}_{(n-m) \ times}]\le 0$.

Now (\ref{tm2_rez1}) easily follows from (\ref{lema2}).
\end{proof}

Following result provides us with a similar upper bound for the difference in the Edmundson-Lah-Ribari\v c inequality, and it is obtained from Lemma \ref{lema3}.

\begin{tm}\label{tm4}
Let $L$ satisfy conditions (L1) and (L2) on a non-empty set $E$ and let $A$ be any  positive linear functional on $L$ with $A(\boldsymbol{1})=1$. Let $f\in \mathcal{C}^n([a,b])$, and let $g\in L$ be any function such that $f\circ g\in L$. If the function $f$ is $n$-convex and if $m\ge 3$ is odd, then 
\begin{align}
    LR(f,g,a,&b,A)\le (b-A(g))\left(f[a,b]-f[b,b]\right)+\sum_{k=2}^{m-1}\dfrac{f^{(k)}(b)}{k!}A[(g-b\boldsymbol{1})^{k}]\nonumber\\
    &+\sum_{k=1}^{n-m}f[\underbrace{b,...,b}_{m \ times};\underbrace{a,...,a}_{k \ times}]A[(g-b\boldsymbol{1})^{m}(g-a\boldsymbol{1})^{k-1}]\label{tm4_rez1}
\end{align}
Inequality (\ref{tm4_rez1}) also holds when the function $f$ is $n$-concave and $m$ is even. In case when the function $f$ is $n$-convex and $m$ is even, or when the function $f$ is $n$-concave and $m$ is odd, the inequality sign in (\ref{tm4_rez1}) is reversed.
\end{tm}

\begin{proof} Similarly as in the proof of the previous theorem, we start with the representation of the left side in the Edmundson-Lah-Ribari\v c inequality (\ref{lema3_rez3}) from Lemma \ref{lema3} with a special focus on the last term:
\begin{align*}
    A(R^{\ast}_m(g))&=A\left(f[g;\underbrace{b\boldsymbol{1},...,b\boldsymbol{1}}_{m \ times};\underbrace{a\boldsymbol{1},...,a\boldsymbol{1}}_{(n-m) \ times}](g-b\boldsymbol{1})^{m}(g-a\boldsymbol{1})^{n-m}\right)
\end{align*}
As before, because of the positivity of the linear functional $A$, we only need to study the sign of the expression:
$$(g(t)-b)^{m}(g(t)-a)^{n-m}f[g(t);\underbrace{b,...,b}_{m \ times};\underbrace{a,a,...,a}_{(n-m) \ times}].$$

Since $a\le g(t)\le b$ for every $t\in E$, we have $\left(g(t)-a\right)^{n-m}\ge 0$ for every $t\in E$ and any choice of $m$. For the same reason we have $(g(t)-b)\le 0$. Trivially it follows that $(g(t)-b)^{m}\le 0$ when $m$ is odd, and $(g(t)-b)^{m}\ge 0$ when $m$ is even.

If the function $f$ is $n$-convex, then its $n$-th order divided differences are greater of equal to zero, and if the function $f$ is $n$-concave, then its $n$-th order divided differences are less or equal to zero.

Now (\ref{tm4_rez1}) easily follows from Lemma (\ref{lema3}).
\end{proof}

\begin{kor}\label{kor1}
Let $L$ satisfy conditions (L1) and (L2) on a non-empty set $E$ and let $A$ be any  positive linear functional on $L$ with $A(\boldsymbol{1})=1$. Let $n$ be an odd number, let $f\in \mathcal{C}^n([a,b])$, and let $g\in L$ be any function such that $f\circ g\in L$. If the function $f$ is $n$-convex and if $m\ge 3$ is odd, then 
\begin{align}
&(A(g)-a)\left(f[a,a]-f[a,b]\right)+\sum_{k=2}^{m-1}\dfrac{f^{(k)}(a)}{k!}A\left[(g-a\boldsymbol{1})^{k}\right]\nonumber\\
    & \ \ \ \ +\sum_{k=1}^{n-m}f[\underbrace{a,...,a}_{m \ times};\underbrace{b,...,b}_{k \ times}]A\left[(g-a\boldsymbol{1})^{m}(g-b\boldsymbol{1})^{k-1}\right]\nonumber\\
    \le LR(&f,g,a,b,A)\le (b-A(g))\left(f[a,b]-f[b,b]\right)+\sum_{k=2}^{m-1}\dfrac{f^{(k)}(b)}{k!}A[(g-b\boldsymbol{1})^{k}]\nonumber\\
    & \ \ \ \ +\sum_{k=1}^{n-m}f[\underbrace{b,...,b}_{m \ times};\underbrace{a,...,a}_{k \ times}]A[(g-b\boldsymbol{1})^{m}(g-a\boldsymbol{1})^{k-1}].\label{kor1_rez1}
\end{align}
Inequality (\ref{kor1_rez1}) also holds when the function $f$ is $n$-concave and $m$ is even. In case when the function $f$ is $n$-convex and $m$ is even, or when the function $f$ is $n$-concave and $m$ is odd, the inequality signs in (\ref{kor1_rez1}) are reversed.
\end{kor}

\begin{rem}
In \cite[Theorem 2.3]{prvi3conv_eksp} is proved that for a 3-convex functions we have
 \begin{align*}
 &(A(g)-a)\left[f'(a)-\dfrac{f(b)-f(a)}{b-a}\right]+\dfrac{f''(a)}{2}A[(g-a\boldsymbol{1})^2], \nonumber\\
\le  &LR(f,g,a,b,A)\le   (b-A(g))\left[\dfrac{f(b)-f(a)}{b-a}-f'(b)\right]+\dfrac{f''(b)}{2}A[(b\boldsymbol{1}-g)^2] \nonumber
 \end{align*}
 and if the function $f$ is 3-concave, then the inequality signs are reversed. 
It is obvious that inequalities (\ref{kor1_rez1}) from Corollary \ref{kor1} provide us with a generalization of the result stated above.
\end{rem}

Next result gives us an upper and a lower bound for the difference in the Edmundson-Lah-Ribari\v c inequality expressed by Hermite's interpolating polynomials in terms of divided differences, and it is obtained from Lemma \ref{lema2}.

\begin{tm}\label{tm3}
Let $L$ satisfy conditions (L1) and (L2) on a non-empty set $E$ and let $A$ be any  positive linear functional on $L$ with $A(\boldsymbol{1})=1$. Let $f\in \mathcal{C}^n([a,b])$, and let $g\in L$ be any function such that $f\circ g\in L$. If the function $f$ is $n$-convex and if $n$ is odd, then 
\begin{align}
    &\sum_{k=2}^{n-1}f[a;\underbrace{b,...,b}_{k \ times}]A\left[(g-a\boldsymbol{1})(g-b\boldsymbol{1})^{k-1}\right]\le LR(f,g,a,b,A)\label{tm3_rez1}\\
    \le &f[a,a;b]A[(g-a\boldsymbol{1})(g-b\boldsymbol{1})]
+\sum_{k=2}^{n-2}f[a,a;\underbrace{b,...,b}_{k \ times}]A\left[(g-a\boldsymbol{1})^2(g-b\boldsymbol{1})^{k-1}\right].\nonumber
\end{align}
Inequalities (\ref{tm3_rez1}) also hold when the function $f$ is $n$-concave and $n$ is even. In case when the function $f$ is $n$-convex and $n$ is even, or when the function $f$ is $n$-concave and $n$ is odd, the inequality signs in (\ref{tm3_rez1}) are reversed.
\end{tm}

\begin{proof} From the discussion about positivity and negativity of the term $A(R_m(g))$ in the proof of Theorem \ref{tm2}, for $m=1$ it follows that 
\begin{itemize}
    \item [$\ast$] $A(R_1(g))\ge 0$ when the function $f$ is $n$-convex and $n$ is odd, or when $f$ is $n$-concave and $n$ even;
    \item [$\ast$] $A(R_1(g))\le 0$ when the function $f$ is $n$-concave and $n$ is odd, or when $f$ is $n$-convex and $n$ even.
\end{itemize} 
Now the identity (\ref{lema2_rez1}) gives us
\begin{align*}
    LR(f,g,a,b,A)\ge &f[a;b,b]A[(g-a\boldsymbol{1})(g-b\boldsymbol{1})]+f[a;b,b,b]A[(g-a\boldsymbol{1})(g-b\boldsymbol{1})^2]\nonumber \\
    &+...+f[a;\underbrace{b,b,...,b}_{(n-1) \ times}]A\left[(g-a\boldsymbol{1})(g-b\boldsymbol{1})^{n-2}\right]
\end{align*}
for $A(R_1(g))\ge 0$, and in case $A(R_1(g))\le 0$ the inequality sign is reversed.

In the same manner, for $m=2$ it follows that 
\begin{itemize}
    \item [$\ast$] $A(R_2(g))\le 0$ when the function $f$ is $n$-convex and $n$ is odd, or when $f$ is $n$-concave and $n$ even;
    \item [$\ast$] $A(R_2(g))\ge 0$ when the function $f$ is $n$-concave and $n$ is odd, or when $f$ is $n$-convex and $n$ even.
\end{itemize} 
In this case the identity (\ref{lema2_rez2}) for $A(R_2(g))\le 0$ gives us
\begin{align*}
    LR(f,g,a,b,A)\le &f[a,a;b]A[(g-a\boldsymbol{1})(g-b\boldsymbol{1})]+f[a,a;b,b]A[(g-a\boldsymbol{1})^2(g-b\boldsymbol{1})]\nonumber \\
&+...+f[a,a;\underbrace{b,b,...,b}_{(n-2) \ times}]A\left[(g-a\boldsymbol{1})^2(g-b\boldsymbol{1})^{n-3}\right]
\end{align*}
and in case $A(R_2(g))\ge 0$ the inequality sign is reversed.

When we combine the two results from above, we get exactly (\ref{tm3_rez1}).
\end{proof}

By utilizing Lemma \ref{lema3} we can get similar bounds for the difference in the Edmundson-Lah-Ribari\v c inequality that hold for all $n\in \mathbb{N}$, not only the odd ones.

\begin{tm}\label{tm5}
Let $L$ satisfy conditions (L1) and (L2) on a non-empty set $E$ and let $A$ be any  positive linear functional on $L$ with $A(\boldsymbol{1})=1$. Let $f\in \mathcal{C}^n([a,b])$, and let $g\in L$ be any function such that $f\circ g\in L$. If the function $f$ is $n$-convex, then 
\begin{align}
  &f[b,b;a]A[(g-b\boldsymbol{1})(g-a\boldsymbol{1})]
+\sum_{k=2}^{n-2}f[b,b;\underbrace{a,...,a}_{k \ times}]A[(g-b\boldsymbol{1})^2(g-a\boldsymbol{1})^{k-1}]\nonumber\\
   \le &LR(f,g,a,b,A)\le 
   \sum_{k=1}^{n-1}f[b;\underbrace{a,...,a}_{k \ times}]A[(g-b\boldsymbol{1})(g-a\boldsymbol{1})^{k-1}].\label{tm5_rez1}
\end{align}
If the function $f$ is $n$-concave, the inequality signs in (\ref{tm5_rez1}) are reversed.
\end{tm}

\begin{proof} We return to the discussion about positivity and negativity of the term $A(R^{\ast}_m(g))$ in the proof of Theorem \ref{tm4}. For $m=1$ we have
$$(g(t)-b)^1(g(t)-a)^{n-1}\le 0 \ \ \mathrm{for \ every} \ t\in E,$$
so $A(R^{\ast}_1(g))\ge 0$ when the function $f$ is $n$-concave, and $A(R^{\ast}_1(g))\le 0$ when the function $f$ is $n$-convex.
Now the identity (\ref{lema3_rez1}) for a $n$-convex function $f$ gives us
\begin{align*}
    LR(f,g,a,b,A)\ge &f[b,b;a]A[(g-b\boldsymbol{1})(g-a\boldsymbol{1})]+f[b,b;a,a]A[(g-b\boldsymbol{1})^2(g-a\boldsymbol{1})]\nonumber \\
& \ \ \ \ +...+f[b,b;\underbrace{a,a,...,a}_{(n-2) \ times}]A[(g-b\boldsymbol{1})^2(g-a\boldsymbol{1})^{n-3}]
\end{align*}
 and if the function $f$ is $n$-concave, the inequality sign is reversed.

Similarly, for $m=2$ we have
$$(g(t)-b)^2(g(t)-a)^{n-2}\ge 0 \ \ \mathrm{for \ every} \ t\in E,$$
so $A(R^{\ast}_2(g))\ge 0$ when the function $f$ is $n$-convex, and $A(R^{\ast}_2(g))\le 0$ when the function $f$ is $n$-concave.
In this case the identity (\ref{lema3_rez2}) for a $n$-convex function $f$ gives us
\begin{align*}
    LR(f,g,a,b,A)\le &f[b;a,a]A[(g-b\boldsymbol{1})(g-a\boldsymbol{1})]+f[b;a,a,a]A[(g-b\boldsymbol{1})(g-a\boldsymbol{1})^2]\nonumber \\
    & \ \ \ \ +...+f[b;\underbrace{a,a,...,a}_{(n-1) \ times}]A[(g-b\boldsymbol{1})(g-a\boldsymbol{1})^{n-2}]
\end{align*}
 and if the function $f$ is $n$-concave, the inequality sign is reversed.

When we combine the two results from above, we get exactly (\ref{tm5_rez1}).
\end{proof}

\begin{rem}
Since
\begin{align*}
    f[a;b,b]&=\dfrac{1}{b-a}\left(f'(b)-\dfrac{f(b)-f(a)}{b-a}\right) \\ f[a,a;b]&=\dfrac{1}{b-a}\left(f'(b)-\dfrac{f(b)-f(a)}{b-a}\right),
\end{align*}
when we take $n=3$ in (\ref{tm3_rez1}) or (\ref{tm5_rez1}), we get that
\begin{align}
    &\dfrac{A[(g-a\boldsymbol{1})(g-b\boldsymbol{1})]}{b-a}\left(f'(b)-\dfrac{f(b)-f(a)}{b-a}\right)\label{stari1}\\
    \le &LR(f,g,a,b,A) \le \dfrac{A[(g-a\boldsymbol{1})(g-b\boldsymbol{1})]}{b-a}\left(f'(b)-\dfrac{f(b)-f(a)}{b-a}\right)\nonumber
\end{align}
holds for a 3-convex function, and for a 3-concave function the inequality signs are reversed. Inequalities (\ref{stari1}) are proved in \cite[Theorem 2.1]{prvi3conv_eksp}, so it follows that Theorem \ref{tm3} and Theorem \ref{tm5} give a generalization of a result from (\cite{prvi3conv_eksp}).
\end{rem}

\section{Applications to Csisz\' ar divergence}\label{sec3}

Let us denote the set of all probability distributions by $\mathbb{P}$, that is we say $\boldsymbol{p}=(p_1,...,p_r)\in \mathbb{P}$ if $p_i\in [0,1]$ for $i=1,...,r$ and $\sum_{i=1}^rp_i=1$. 

Numerous theoretic divergence measures between two
probability distributions have been introduced and comprehensively studied. Their applications can be found in the analysis of contingency tables \cite{10}, in approximation of probability distributions \cite{6}, \cite{21}, in signal processing \cite{14}, and in pattern recognition \cite{3}, \cite{4}.

Csisz\' ar \cite{csiszar1}-\cite{csiszar2} introduced the $f-$divergence functional as
\begin{equation}\label{fdiv}
D_f(\boldsymbol{p},\boldsymbol{q})=\sum_{i=1}^rq_if\left(\frac{p_i}{q_i}\right),
\end{equation}
where $f\colon [0,+\infty \rangle$ is a convex function, and it represent a "distance function" on the set of probability distributions $\mathbb{P}$.

A great number of theoretic divergences are special cases of Csisz\' ar $f$-divergence for different choices of the function $f$.

As in Csisz\' ar \cite{csiszar2}, we interpret undefined expressions by
$$f(0)=\lim_{t \to 0^+}f(t), \ \ 0\cdot f\left(\dfrac{0}{0}\right)=0,$$
$$0\cdot f\left(\dfrac{a}{0}\right)=\lim_{\epsilon \to 0^+}f\left(\dfrac{a}{\epsilon}\right)=a\cdot \lim_{t \to \infty}\dfrac{f(t)}{t}.$$

In this section our intention is to derive mutual bounds for the generalized $f$-divergence functional in described setting. In such a way, we will obtain 
some new reverse relations for the generalized $f$-divergence functional that correspond to the class of $n$-convex functions. It is a generalization of the results obtained in \cite{prvi3conv_eksp}. Throughout this section, when mentioning the interval $[a,b]$, we assume that $[a,b]\subseteq \mathbb{R}_+$. For a $n$-convex function $f\colon [m,M]\to \mathbb{R}$ we give the following definition of generalized $f$-divergence functional:
\begin{equation}\label{fdivg}
\tilde{D}_f(\boldsymbol{p},\boldsymbol{q})=\sum_{i=1}^rq_if\left(\frac{p_i}{q_i}\right).
\end{equation}

The first result in this section is carried out by virtue of our Theorem \ref{tm2}.

\begin{tm}\label{cd_tm2}
Let $[a,b]\subset \mathbb{R}$ be an interval such that $a\le 1\le b$. Let $f\in \mathcal{C}^n([a,b])$ and let $\boldsymbol{p}=(p_1,...,p_r)$ and $\boldsymbol{p}=(q_1,...,q_r)$ be probability distributions such that $p_i/q_i\in [a,b]$ for every $i=1,...,r$. If the function $f$ is $n$-convex and if $n$ and $3\le m\le n-1$ are of different parity, then 
\begin{align}
   &\dfrac{b-1}{b-a}f(a)+\dfrac{1-a}{b-a}f(b)-\tilde{D}_f(\boldsymbol{p},\boldsymbol{q})\\
    \le &\left(1-a\right)\left(f[a,a]-f[a,b]\right)+\sum_{k=2}^{m-1}\dfrac{f^{(k)}(a)}{k!}\sum_{i=1}^r\dfrac{(p_i-aq_i)^{k}}{q_i^{k-1}}\nonumber\\
    &+\sum_{k=1}^{n-m}f[\underbrace{a,...,a}_{m \ times};\underbrace{b,...,b}_{k \ times}]\sum_{i=1}^r\dfrac{(p_i-aq_i)^{m}(p_i-aq_i)^{k-1}}{q_i^{m+k-2}}. \label{cd_tm2_rez1}
\end{align}
Inequality (\ref{cd_tm2_rez1}) also holds when the function $f$ is $n$-concave and $n$ and $m$ are of equal parity. In case when the function $f$ is $n$-convex and $n$ and $m$ are of equal parity, or when the function $f$ is $n$-concave and $n$ and $m$ are of different parity, the inequality sign in (\ref{cd_tm2_rez1}) is reversed.
\end{tm}

\begin{proof} Let $\boldsymbol{x}=(x_1,...,x_r)$ be such that $x_i\in [a,b]$ for $i=1,...,r$. In the relation (\ref{tm2_rez1}) we can replace
$$g\longleftrightarrow \boldsymbol{x}, \ \ \mathrm{and} \ \  A(\boldsymbol{x})=\sum_{i=1}^rp_ix_i.$$
In that way we get 
\begin{align*}
&\dfrac{b-\bar{x}}{b-a}f(a)+\dfrac{\bar{x}-a}{b-a}f(b)-\sum_{i=1}^rp_if(x_i)\\
    \le &\left(\bar{x}-a\right)\left(f[a,a]-f[a,b]\right)+\sum_{k=2}^{m-1}\dfrac{f^{(k)}(a)}{k!}\sum_{i=1}^rp_i(x_i-a)^{k}\nonumber\\
    &+\sum_{k=1}^{n-m}f[\underbrace{a,...,a}_{m \ times};\underbrace{b,...,b}_{k \ times}]\sum_{i=1}^rp_i(x_i-a)^{m}(x_i-b)^{k-1}, 
\end{align*}
where $\bar{x}=\sum_{i=1}^np_ix_i$. In the previous relation we can set
$$p_i=q_i \ \ \mathrm{and} \ \ x_i=\dfrac{p_i}{q_i},$$
and after calculating 
$$\bar{x}=\sum_{i=1}^nq_i\dfrac{p_i}{q_i}=\sum_{i=1}^np_i=1$$
we get (\ref{cd_tm2_rez1}).
\end{proof}

By utilizing Theorem \ref{tm4} in the analogous way as above, we get an Edmundson-Lah-Ribari\v c type inequality for the generalized $f$-divergence functional (\ref{fdivg}) which does not depend on parity of $n$, and it is given in the following theorem.

\begin{tm}\label{cd_tm4}
Let $[a,b]\subset \mathbb{R}$ be an interval such that $a\le 1\le b$. Let $f\in \mathcal{C}^n([a,b])$ and let $\boldsymbol{p}=(p_1,...,p_r)$ and $\boldsymbol{p}=(q_1,...,q_r)$ be probability distributions such that $p_i/q_i\in [a,b]$ for every $i=1,...,r$. If the function $f$ is $n$-convex and if $3\le m\le n-1$ is odd, then 
\begin{align}
    &\dfrac{b-1}{b-a}f(a)+\dfrac{1-a}{b-a}f(b)-\tilde{D}_f(\boldsymbol{p},\boldsymbol{q})\nonumber \\
   \le &(b-1)\left(f[a,b]-f[b,b]\right)+\sum_{k=2}^{m-1}\dfrac{f^{(k)}(b)}{k!}\sum_{i=1}^r\dfrac{(p_i-bq_i)^{k}}{q_i^{k-1}}\nonumber\\
    &+\sum_{k=1}^{n-m}f[\underbrace{b,...,b}_{m \ times};\underbrace{a,...,a}_{k \ times}]\sum_{i=1}^r\dfrac{(p_i-bq_i)^{m}(p_i-aq_i)^{k-1}}{q_i^{m+k-2}}\label{cd_tm4_rez1}
\end{align}
Inequality (\ref{cd_tm4_rez1}) also holds when the function $f$ is $n$-concave and $m$ is even. In case when the function $f$ is $n$-convex and $m$ is even, or when the function $f$ is $n$-concave and $m$ is odd, the inequality sign in (\ref{cd_tm4_rez1}) is reversed.
\end{tm}

Another generalization of the Edmundson-Lah-Ribari\v c inequality, which provides us with a lower and an upper bound for the generalized $f$-divergence functional, is given in the following theorem.

\begin{tm}\label{cd_tm3}
Let $[a,b]\subset \mathbb{R}$ be an interval such that $a\le 1\le b$. Let $f\in \mathcal{C}^n([a,b])$ and let $\boldsymbol{p}=(p_1,...,p_r)$ and $\boldsymbol{p}=(q_1,...,q_r)$ be probability distributions such that $p_i/q_i\in [a,b]$ for every $i=1,...,r$. If the function $f$ is $n$-convex and if $n$ is odd, then we have
\begin{align}
    & \sum_{k=2}^{n-1}f[a;\underbrace{b,b,...,b}_{k \ times}]\sum_{i=1}^r\dfrac{(p_i-aq_i)(p_i-bq_i)^{k-1}}{q_i^{k-1}}
    \le \dfrac{b-1}{b-a}f(a)+\dfrac{1-a}{b-a}f(b)-\tilde{D}_f(\boldsymbol{p},\boldsymbol{q})\nonumber\\
    \le &f[a,a;b]\sum_{i=1}^r\dfrac{(p_i-aq_i)(p_i-bq_i)}{q_i}+\sum_{k=2}^{n-2}f[a,a;\underbrace{b,...,b}_{k \ times}]\sum_{i=1}^r\dfrac{(p_i-aq_i)^2(p_i-bq_i)^{k-1}}{q_i^k}.\label{cd_tm3_rez1}
\end{align}
Inequalities (\ref{cd_tm3_rez1}) also hold when the function $f$ is $n$-concave and $n$ is even. In case when the function $f$ is $n$-convex and $n$ is even, or when the function $f$ is $n$-concave and $n$ is odd, the inequality signs in (\ref{cd_tm3_rez1}) are reversed.
\end{tm}

\begin{proof}
We start with inequalities (\ref{tm3_rez1}) from Theorem \ref{tm3}, and follow the steps from the proof of Theorem \ref{cd_tm2}.
\end{proof}

By utilizing Theorem \ref{tm5} in an analogue way, we can get similar bounds for the generalized $f$-divergence functional that hold for all $n\in \mathbb{N}$, not only the odd ones.

\begin{tm}\label{cd_tm5}
Let $[a,b]\subset \mathbb{R}$ be an interval such that $a\le 1\le b$. Let $f\in \mathcal{C}^n([a,b])$ and let $\boldsymbol{p}=(p_1,...,p_r)$ and $\boldsymbol{p}=(q_1,...,q_r)$ be probability distributions such that $p_i/q_i\in [a,b]$ for every $i=1,...,r$. If the function $f$ is $n$-convex, then we have
\begin{align}
  &f[b,b;a]\sum_{i=1}^r\dfrac{(p_i-aq_i)(p_i-bq_i)}{q_i}+\sum_{k=2}^{n-2}f[b,b;\underbrace{a,a,...,a}_{k \ times}]\sum_{i=1}^r\dfrac{(p_i-aq_i)^{k-1}(p_i-bq_i)^2}{q_i^k}\nonumber \\
    \le &\dfrac{b-1}{b-a}f(a)+\dfrac{1-a}{b-a}f(b)-\tilde{D}_f(\boldsymbol{p},\boldsymbol{q})
    \le  \sum_{k=2}^{n-1}f[b;\underbrace{a,...,a}_{k \ times}]\sum_{i=1}^r\dfrac{(p_i-aq_i)^{k-1}(p_i-bq_i)}{q_i^{k-1}}.\label{cd_tm5_rez1}
\end{align}
If the function $f$ is $n$-concave, the inequality signs in (\ref{cd_tm5_rez1}) are reversed.
\end{tm}

\begin{primjer}\label{primjer}
Let $\boldsymbol{p}=(p_1,...,p_r)$ and $\boldsymbol{p}=(q_1,...,q_r)$ be probability distributions.
\begin{itemize}
\item[$\triangleright$] \textbf{Kullback-Leibler divergence} of the probability distributions $\boldsymbol{p}$ and $\boldsymbol{q}$ is defined as
$$D_{KL}(\boldsymbol{p},\boldsymbol{q})=\sum_{i=1}^rq_i\log \dfrac{q_i}{p_i},$$
and the corresponding generating function is $f(t)=t\log t, t>0$. We can calculate 
$$f^{(n)}(t)=(-1)^n(n-2)!t^{-(n-1)}.$$
It is clear that this function is $(2n-1)$-concave and $(2n)$-convex for any $n\in \mathbb{N}$.

\item[$\triangleright$] \textbf{Hellinger divergence} of the probability distributions $\boldsymbol{p}$ and $\boldsymbol{q}$ is defined as
$$D_{H}(\boldsymbol{p},\boldsymbol{q})=\dfrac{1}{2}\sum_{i=1}^n(\sqrt{q_i}-\sqrt{p_i})^2,$$
and the corresponding generating function is $f(t)=\frac{1}{2}(1-\sqrt{t})^2, t>0$. We see that 
$$f^{(n)}(t)=(-1)^n\frac{(2n-3)!!}{2^n}t^{-\frac{2n-1}{2}},$$ 
so function $f$ is $(2n-1)$-concave and $(2n)$-convex for any $n\in \mathbb{N}$.

\item[$\triangleright$] \textbf{Harmonic divergence} of the probability distributions $\boldsymbol{p}$ and $\boldsymbol{q}$ is defined as
$$D_{Ha}(\boldsymbol{p},\boldsymbol{q})=\sum_{i=1}^n\dfrac{2p_iq_i}{p_i+q_i},$$
and the corresponding generating function is $f(t)=\frac{2t}{1+t}$. We can calculate
$$f^{(n)}(t)=2(-1)^{n+1}n!(1+t)^{-(n+1)}.$$
Two cases need to be considered:
\begin{itemize}
    \item[$\ast$] if $t<-1$, then the function $f$ is $n$-convex for every $n\in \mathbb{N}$;
    \item[$\ast$] if $t>-1$, then the function $f$ is $(2n)$-concave and $(2n-1)$-convex for any $n\in \mathbb{N}$.
\end{itemize}

\item[$\triangleright$] \textbf{Jeffreys divergence} of the probability distributions $\boldsymbol{p}$ and $\boldsymbol{q}$ is defined as
$$D_{J}(\boldsymbol{p},\boldsymbol{q})=\dfrac{1}{2}\sum_{i=1}^n(q_i-p_i)\log \dfrac{q_i}{p_i},$$
and the corresponding generating function is $f(t)=(1-t)\log \frac{1}{t}, t>0$. After calculating, we see that 
$$f^{(n)}(t)=(-1)^{n+1}t^{-n}(n-1)!(1+nt).$$
Obviously, this function is $(2n-1)$-convex and $(2n)$-concave for any $n\in \mathbb{N}$.
\end{itemize}
It is clear that all of the results from this section can be applied to the special types of divergences mentioned in this example.
\end{primjer}

\section{Examples with Zipf and Zipf-Mandelbrot law}\label{sec4}

Zipf’s law \cite{zipf}, \cite{zipf1} has a significant application in a wide variety of scientific disciplines - from astronomy to demographics to software
structure to economics to zoology, and even to warfare \cite{war}. It is one of the basic laws in information science and bibliometrics, but it is also often used in linguistics.
Typically one is dealing with integer-valued observables (numbers of objects, people, cities, words, animals, corpses) and the frequency of their occurrence.

Probability mass function of Zipf's law with parameters $N\in \mathbb{N}$ and $s>0$ is:
\begin{equation*}
f(k;N,s)=\frac{1/k^s}{H_{N,s}}, \ \ \mathrm{where} \ \ H_{N,s}=\sum_{i=1}^N\frac{1}{i^s}.
\end{equation*}

Benoit Mandelbrot in 1966 gave an improvement of Zipf law for the count of the low-rank words. Various scientific fields use this law for different purposes, for example information sciences use it for indexing \cite{egg,sila}, ecological field studies in predictability of ecosystem \cite{moi}, in music it is used to determine aesthetically pleasing music \cite{mana}.

Zipf–Mandelbrot law is a discrete probability distribution
with parameters $N\in \mathbb{N}$, $q,s\in \mathbb{R}$ such that $q\ge 0$ and $s>0$, possible values $\{1, 2, ..., N\}$ and probability mass function
\begin{equation}\label{ZM}
f(i;N,q,s)=\frac{1/(i+q)^s}{H_{N,q,s}}, \ \ \mathrm{where} \ \ H_{N,q,s}=\sum_{i=1}^N\frac{1}{(i+q)^s}.
\end{equation}

Let $\boldsymbol{p}$ and $\boldsymbol{q}$ be Zipf-Mandelbrot laws with parameters $N\in \mathbb{N}$, $q_1,q_2\ge 0$ and $s_1,s_2>0$ respectively and let us denote
\begin{align}
H_{N,q_1,s_1}&=H_1, \ H_{N,q_2,s_2}=H_2\nonumber\\
a_{\boldsymbol{p},\boldsymbol{q}}:&=\mathrm{min}\left\{\dfrac{p_i}{q_i}\right\}=\dfrac{H_2}{H_1}\mathrm{min}\left\{\dfrac{(i+q_2)^{s_2}}{(i+q_1)^{s_1}}\right\}\nonumber\\
b_{\boldsymbol{p},\boldsymbol{q}}:&=\mathrm{max}\left\{\dfrac{p_i}{q_i}\right\}=\dfrac{H_2}{H_1}\mathrm{max}\left\{\dfrac{(i+q_2)^{s_2}}{(i+q_1)^{s_1}}\right\}\label{defmMKL}
\end{align}

In this section we utilize the results regarding Csisz\' ar divergence from the previous section in order to obtain different inequalities for the Zipf-Mandelbrot law. The following results are special cases of Theorems \ref{cd_tm2}, \ref{cd_tm4}, \ref{cd_tm3} and \ref{cd_tm5} respectively, and they gives us Edmundson-Lah-Ribari\v c type inequality for the generalized $f$-divergence of the Zipf–Mandelbrot law.

\begin{kor}\label{zm_kor2}
Let $\boldsymbol{p}$ and $\boldsymbol{q}$ be Zipf-Mandelbrot laws with parameters $N\in \mathbb{N}$, $q_1,q_2\ge 0$ and $s_1,s_2>0$ respectively, and let $H_1$, $H_2$,  $a_{\boldsymbol{p},\boldsymbol{q}}$ and $a_{\boldsymbol{p},\boldsymbol{q}}$ be defined in (\ref{defmMKL}). Let Let $f\in \mathcal{C}^n([a_{\boldsymbol{p},\boldsymbol{q}},b_{\boldsymbol{p},\boldsymbol{q}}])$ be a $n$-convex function. If $n$ and $3\le m\le n-1$ are of different parity, then 
\begin{align*}
   &\dfrac{b_{\boldsymbol{p},\boldsymbol{q}}-1}{b_{\boldsymbol{p},\boldsymbol{q}}-a_{\boldsymbol{p},\boldsymbol{q}}}f(a_{\boldsymbol{p},\boldsymbol{q}})+\dfrac{1-a_{\boldsymbol{p},\boldsymbol{q}}}{b_{\boldsymbol{p},\boldsymbol{q}}-a_{\boldsymbol{p},\boldsymbol{q}}}f(b_{\boldsymbol{p},\boldsymbol{q}})-\tilde{D}_f(\boldsymbol{p},\boldsymbol{q})\\
    \le &\left(1-a_{\boldsymbol{p},\boldsymbol{q}}\right)\left(f'(a_{\boldsymbol{p},\boldsymbol{q}})-f[a_{\boldsymbol{p},\boldsymbol{q}},b_{\boldsymbol{p},\boldsymbol{q}}]\right)+\sum_{k=2}^{m-1}\dfrac{f^{(k)}(a_{\boldsymbol{p},\boldsymbol{q}})}{H_2k!}\sum_{i=1}^r\dfrac{\left(\frac{H_2(i+q_2)^{s_2}}{H_1(i+q_1)^{s_1}}-a_{\boldsymbol{p},\boldsymbol{q}}\right)^{k}}{(i+q_2)^{s_2}}\nonumber\\
    &+\sum_{k=1}^{n-m}f[\underbrace{a_{\boldsymbol{p},\boldsymbol{q}},...,a_{\boldsymbol{p},\boldsymbol{q}}}_{m \ times};\underbrace{b_{\boldsymbol{p},\boldsymbol{q}},...,b_{\boldsymbol{p},\boldsymbol{q}}}_{k \ times}]\sum_{i=1}^r\dfrac{\left(\frac{H_2(i+q_2)^{s_2}}{H_1(i+q_1)^{s_1}}-a_{\boldsymbol{p},\boldsymbol{q}}\right)^{m}\left(\frac{H_2(i+q_2)^{s_2}}{H_1(i+q_1)^{s_1}}-b_{\boldsymbol{p},\boldsymbol{q}}\right)^{k-1}}{H_2(i+q_2)^{s_2}}. 
\end{align*}
This inequality also holds when the function $f$ is $n$-concave and $n$ and $m$ are of equal parity. In case when the function $f$ is $n$-convex and $n$ and $m$ are of equal parity, or when the function $f$ is $n$-concave and $n$ and $m$ are of different parity, the inequality sign is reversed.
\end{kor}

\begin{kor}\label{zm_kor4}
Let $\boldsymbol{p}$ and $\boldsymbol{q}$ be Zipf-Mandelbrot laws with parameters $N\in \mathbb{N}$, $q_1,q_2\ge 0$ and $s_1,s_2>0$ respectively, and let $H_1$, $H_2$,  $a_{\boldsymbol{p},\boldsymbol{q}}$ and $a_{\boldsymbol{p},\boldsymbol{q}}$ be defined in (\ref{defmMKL}). Let Let $f\in \mathcal{C}^n([a_{\boldsymbol{p},\boldsymbol{q}},b_{\boldsymbol{p},\boldsymbol{q}}])$ be a $n$-convex function and let $3\le m\le n-1$ be of different parity. Then 
\begin{align*}
     &\dfrac{b_{\boldsymbol{p},\boldsymbol{q}}-1}{b_{\boldsymbol{p},\boldsymbol{q}}-a_{\boldsymbol{p},\boldsymbol{q}}}f(a_{\boldsymbol{p},\boldsymbol{q}})+\dfrac{1-a_{\boldsymbol{p},\boldsymbol{q}}}{b_{\boldsymbol{p},\boldsymbol{q}}-a_{\boldsymbol{p},\boldsymbol{q}}}f(b_{\boldsymbol{p},\boldsymbol{q}})-\tilde{D}_f(\boldsymbol{p},\boldsymbol{q})\\
   \le &(b_{\boldsymbol{p},\boldsymbol{q}}-1)\left(f[a_{\boldsymbol{p},\boldsymbol{q}},b_{\boldsymbol{p},\boldsymbol{q}}]-f'(b_{\boldsymbol{p},\boldsymbol{q}})\right)+\sum_{k=2}^{m-1}\dfrac{f^{(k)}(b_{\boldsymbol{p},\boldsymbol{q}})}{H_2k!}\sum_{i=1}^r\dfrac{\left(\frac{H_2(i+q_2)^{s_2}}{H_1(i+q_1)^{s_1}}-b_{\boldsymbol{p},\boldsymbol{q}}\right)^{k}}{(i+q_2)^{s_2}}\nonumber\\
     &+\sum_{k=1}^{n-m}f[\underbrace{b_{\boldsymbol{p},\boldsymbol{q}},...,b_{\boldsymbol{p},\boldsymbol{q}}}_{m \ times};\underbrace{a_{\boldsymbol{p},\boldsymbol{q}},...,a_{\boldsymbol{p},\boldsymbol{q}}}_{k \ times}]\sum_{i=1}^r\dfrac{\left(\frac{H_2(i+q_2)^{s_2}}{H_1(i+q_1)^{s_1}}-b_{\boldsymbol{p},\boldsymbol{q}}\right)^{m}\left(\frac{H_2(i+q_2)^{s_2}}{H_1(i+q_1)^{s_1}}-a_{\boldsymbol{p},\boldsymbol{q}}\right)^{k-1}}{H_2(i+q_2)^{s_2}}.
\end{align*}
The inequality above also holds when the function $f$ is $n$-concave and $m$ is even. In case when the function $f$ is $n$-convex and $m$ is even, or when the function $f$ is $n$-concave and $m$ is odd, the inequality sign is reversed.
\end{kor}

\begin{kor}\label{zm_kor3}
Let $\boldsymbol{p}$ and $\boldsymbol{q}$ be Zipf-Mandelbrot laws with parameters $N\in \mathbb{N}$, $q_1,q_2\ge 0$ and $s_1,s_2>0$ respectively, and let $H_1$, $H_2$,  $a_{\boldsymbol{p},\boldsymbol{q}}$ and $a_{\boldsymbol{p},\boldsymbol{q}}$ be defined in (\ref{defmMKL}). Let Let $f\in \mathcal{C}^n([a_{\boldsymbol{p},\boldsymbol{q}},b_{\boldsymbol{p},\boldsymbol{q}}])$ be a $n$-convex function. If $n$ is odd, then we have
\begin{align*}
    &\sum_{k=2}^{n-1}f[a_{\boldsymbol{p},\boldsymbol{q}};\underbrace{b_{\boldsymbol{p},\boldsymbol{q}},...,b_{\boldsymbol{p},\boldsymbol{q}}}_{k \ times}]\sum_{i=1}^r\dfrac{\left(\frac{H_2(i+q_2)^{s_2}}{H_1(i+q_1)^{s_1}}-a_{\boldsymbol{p},\boldsymbol{q}}\right)\left(\frac{H_2(i+q_2)^{s_2}}{H_1(i+q_1)^{s_1}}-b_{\boldsymbol{p},\boldsymbol{q}}\right)^{k-1}}{H_2(i+q_2)^{s_2}}\\
    \le      &\dfrac{b_{\boldsymbol{p},\boldsymbol{q}}-1}{b_{\boldsymbol{p},\boldsymbol{q}}-a_{\boldsymbol{p},\boldsymbol{q}}}f(a_{\boldsymbol{p},\boldsymbol{q}})+\dfrac{1-a_{\boldsymbol{p},\boldsymbol{q}}}{b_{\boldsymbol{p},\boldsymbol{q}}-a_{\boldsymbol{p},\boldsymbol{q}}}f(b_{\boldsymbol{p},\boldsymbol{q}})-\tilde{D}_f(\boldsymbol{p},\boldsymbol{q})\\
    \le &f[a_{\boldsymbol{p},\boldsymbol{q}},a_{\boldsymbol{p},\boldsymbol{q}};b_{\boldsymbol{p},\boldsymbol{q}}]\sum_{i=1}^r\dfrac{\left(\frac{H_2(i+q_2)^{s_2}}{H_1(i+q_1)^{s_1}}-a_{\boldsymbol{p},\boldsymbol{q}}\right)\left(\frac{H_2(i+q_2)^{s_2}}{H_1(i+q_1)^{s_1}}-b_{\boldsymbol{p},\boldsymbol{q}}\right)}{H_2(i+q_2)^{s_2}}\nonumber \\
& \ \ \ \ +\sum_{k=2}^{n-2}f[a_{\boldsymbol{p},\boldsymbol{q}},a_{\boldsymbol{p},\boldsymbol{q}};\underbrace{b_{\boldsymbol{p},\boldsymbol{q}},...,b_{\boldsymbol{p},\boldsymbol{q}}}_{k \ times}]\sum_{i=1}^r\dfrac{\left(\frac{H_2(i+q_2)^{s_2}}{H_1(i+q_1)^{s_1}}-a_{\boldsymbol{p},\boldsymbol{q}}\right)^2\left(\frac{H_2(i+q_2)^{s_2}}{H_1(i+q_1)^{s_1}}-b_{\boldsymbol{p},\boldsymbol{q}}\right)^{k-1}}{H_2(i+q_2)^{s_2}}.
\end{align*}
Stated inequalities also hold when the function $f$ is $n$-concave and $n$ is even. In case when the function $f$ is $n$-convex and $n$ is even, or when the function $f$ is $n$-concave and $n$ is odd, the inequality signs are reversed.
\end{kor}

\begin{kor}\label{zm_kor5}
Let $\boldsymbol{p}$ and $\boldsymbol{q}$ be Zipf-Mandelbrot laws with parameters $N\in \mathbb{N}$, $q_1,q_2\ge 0$ and $s_1,s_2>0$ respectively, and let $H_1$, $H_2$,  $a_{\boldsymbol{p},\boldsymbol{q}}$ and $a_{\boldsymbol{p},\boldsymbol{q}}$ be defined in (\ref{defmMKL}). Let Let $f\in \mathcal{C}^n([a_{\boldsymbol{p},\boldsymbol{q}},b_{\boldsymbol{p},\boldsymbol{q}}])$ be a $n$-convex function. Then we have
\begin{align*}
  &f[b_{\boldsymbol{p},\boldsymbol{q}},b_{\boldsymbol{p},\boldsymbol{q}};a_{\boldsymbol{p},\boldsymbol{q}}]\sum_{i=1}^r\dfrac{\left(\frac{H_2(i+q_2)^{s_2}}{H_1(i+q_1)^{s_1}}-a_{\boldsymbol{p},\boldsymbol{q}}\right)\left(\frac{H_2(i+q_2)^{s_2}}{H_1(i+q_1)^{s_1}}-b_{\boldsymbol{p},\boldsymbol{q}}\right)}{H_2(i+q_2)^{s_2}}\\
& \ \ \ \ +\sum_{k=2}^{n-2}f[b_{\boldsymbol{p},\boldsymbol{q}},b_{\boldsymbol{p},\boldsymbol{q}};\underbrace{a_{\boldsymbol{p},\boldsymbol{q}},...,a_{\boldsymbol{p},\boldsymbol{q}}}_{k \ times}]\sum_{i=1}^r\dfrac{\left(\frac{H_2(i+q_2)^{s_2}}{H_1(i+q_1)^{s_1}}-a_{\boldsymbol{p},\boldsymbol{q}}\right)^{k-1}\left(\frac{H_2(i+q_2)^{s_2}}{H_1(i+q_1)^{s_1}}-b_{\boldsymbol{p},\boldsymbol{q}}\right)^2}{H_2(i+q_2)^{s_2}}\nonumber\\
    \le &\dfrac{b-1}{b-a}f(a)+\dfrac{1-a}{b-a}f(b)-\tilde{D}_f(\boldsymbol{p},\boldsymbol{q})\nonumber\\ 
  \le  &\sum_{k=2}^{n-1}f[b_{\boldsymbol{p},\boldsymbol{q}};\underbrace{a_{\boldsymbol{p},\boldsymbol{q}},...,a_{\boldsymbol{p},\boldsymbol{q}}}_{k \ times}]\sum_{i=1}^r\dfrac{\left(\frac{H_2(i+q_2)^{s_2}}{H_1(i+q_1)^{s_1}}-a_{\boldsymbol{p},\boldsymbol{q}}\right)^{k-1}\left(\frac{H_2(i+q_2)^{s_2}}{H_1(i+q_1)^{s_1}}-b_{\boldsymbol{p},\boldsymbol{q}}\right)}{H_2(i+q_2)^{s_2}}.
\end{align*}
If the function $f$ is $n$-concave, the inequality signs are reversed.
\end{kor}

\begin{rem}
By taking into consideration Example \ref{primjer} one can see that general results from this section can easily be applied to any of the following divergences: Kullback-Leibler divergence, Hellinger divergence, harmonic divergence or Jeffreys divergence.
\end{rem}

\section*{Acknowledgements}

The publication was supported by the Ministry of Education and Science of the Russian Federation (the Agreement
number No. 02.a03.21.0008.)


\begin{thebibliography}{99}

\bibitem{abr} S. Abramovich, \textit{Quasi-arithmetic means and subquadracity}, J. Math. Inequal., \textbf{9} (4), (2015), 1157--1168.

\bibitem{agarwal} R. P. Agarwal, P. J. Y. Wong, \textit{Error Inequalities in Polynomial Interpolation and Their Applications}, Kluwer Academic Publishers, Dordrecht, Boston, London, 1993.

\bibitem{beesack} P. R. Beesack, J. E. Pe\v{c}ari\'{c}, \textit{On the Jessen's inequality for convex functions}, J. Math. Anal. \textbf{110}(1985), 536--552.

\bibitem{3} M. Ben Bassat, \emph{f-entropies, probability of error, and feature selection}, Inform. Contr., {\bf 39}, (1978), 227--242.

\bibitem{bullen} P. S. Bullen, D. S. Mitrinovi\'{c}, P. M. Vasi\'{c}, (1987). \textit{Means and their inequalities}, D. Reidel Publishing Co., Dordrecht, Boston, Lancaster and Tokyo.

\bibitem{4} C. H. Chen, \emph{Statistical Pattern Recognition}, Rochelle Park, NJ: Hayden Book Co., 1973.

\bibitem{hor} D. Choi, M. Krni\' c, J. Pe\v cari\' c, \textit{Improved Jensen-type inequalities via linear interpolation and applications}, J. Math. Inequal., \textbf{11} (2), (2017), 301--322.

\bibitem{6} C. K. Chow, C. N. Liu, \emph{Approximating discrete probability distributions with dependence trees}, IEEE Trans. Inform. Theory, {\bf 14} (3), (1968), 462--467.

\bibitem{csiszar1} I. Csisz\' ar, \textit{Information measures: A critical survey}, Trans. 7th Prague Conf. on Info. Th.
Statist. Decis. Funct., Random Processes and 8th European Meeting of Statist., Volume B,
Academia Prague, 1978, 73--86

\bibitem{csiszar2} I. Csisz\' ar, \textit{Information-type measures of difference of probability functions and indirect  observations}, Studia Sci. Math. Hungar., \textbf{2} (1967), 299--318.

\bibitem{egg} L. Egghe, R. Rousseau, \textit{Introduction to Informetrics. Quantitative Methods in Library, Documentation and Information Science}, Elsevier Science Publishers. New York: 1990.

\bibitem{war} L. Fry Richardson, \emph{Statistics of Deadly Quarrels}, New York: Marcel Dekker, 1960.

\bibitem{10} D. V. Gokhale, S. Kullback, \emph{Information in Contingency Tables}, Pacific Grove, Boxwood Press 1978. 

\bibitem{issa} E. Issacson, H. B. Keller, \textit{Analysis of Numerical methods}, Dover Publications Inc., New York: 1966.

\bibitem{julije} J. Jak\v seti\' c, J. Pe\v cari\' c, \textit{Exponential convexity method}, J. Conv. Anal., \textbf{20} (1), (2013), 181--197.

\bibitem{ELRpaper} R. Jak\v si\' c, J. Pe\v cari\' c, \textit{Levinson's type generalization of the Edmundson-Lah-Ribari\v c inequality,} Mediterr. J. Math., \textbf{13} (1), (2016), 483--496.

\bibitem{jessen} B. Jessen, (1931). \textit{Bemaerkinger om konvekse Funktioner og Uligheder imellem Middelvaerdier I}, Mat. Tidsskrift, B, 17-28.

\bibitem{14} T. Kailath, \emph{The divergence and Bhattacharyya distance measures in signal selection}, IEEE Transuctions Commun. Technol., {\bf 15} (1), (1967), 52--60.

\bibitem{treci} M. Krni\' c, R. Miki\' c, J. Pe\v cari\' c, \textit{Strengthened converses of the Jensen and Edmundson-Lah-Ribari\v c inequalities}, Advances in Operator Theory, \textbf{1} (1), (2016), 104--122.

\bibitem{mon6} K. Kruli\' c Himmelreich, J. Pe\v cari\' c, D. Pokaz, \textit{Inequalities of Hardy and Jensen / New Hardy type 
inequalities with general kernels}, Monographs in inequalities 6, Element, Zagreb, 2013.

\bibitem{jin} J. Liang, G. Shi, \textit{Comparison of differences among power means $Q_{r,\alpha}(a,b,\boldsymbol{x})s$}, J. Math. Inequal., \textbf{9} (2), (2015), 351--360.

\bibitem{21} J. Lin, S. K. M. Wong, \emph{Approximation of discrete probability distributions based on a new divergence measure}, Congressus Numerantiitm, {\bf 61}, (1988), 75--80.

\bibitem{mana} B. Manaris, D. Vaughan, C. S. Wagner, J. Romero, R. B. Davis,  \textit{Evolutionary Music and the Zipf-Mandelbrot Law: Developing Fitness Functions for Pleasant Music}, Proceedings of 1st European Workshop on Evolutionary Music and Art (EvoMUSART2003),  522--534.

\bibitem{prvi3conv} R. Miki\' c, \DJ. Pe\v cari\' c, J. Pe\v cari\' c, \textit{Inequalities of the Jensen and Edmundson-Lah-Ribari\v c type for 3-convex functions with applications}, J. Math. Inequal., to appear

\bibitem{prvi3conv_eksp} R. Miki\' c, \DJ. Pe\v cari\' c, J. Pe\v cari\' c, \textit{Some inequalities of the Edmundson-Lah-Ribari\v c type for 3-convex functions with applications}, submitted

\bibitem{moi} D. Mouillot, A. Lepretre,  \textit{Introduction of relative abundance distribution (RAD) indices, estimated from the rank-frequency diagrams (RFD), to assess changes in community diversity}, Environmental Monitoring and Assessment. Springer. \textbf{63} (2), (2000), 279--295.

\bibitem{hor1} Z. Pavi\' c, \textit{The Jensen and Hermite–Hadamard inequality on the triangle}, J. Math. Inequal., \textbf{11} (4), (2017), 1099--1112.

\bibitem{roq} J. Pe\v cari\' c, I. Peri\' c, G. Roquia, \emph{Exponentially convex functions generated by Wulbert's inequality and Stolarsky-type means}, Math. Comp. Model., {\bf 55}, (2012), 1849--1857.

\bibitem{ivelic} J. Pe\v cari\' c, J. Peri\' c \emph{New improvement of the converse Jensen inequality}, Math. Inequal. Appl., {\bf 21} (1), (2018), 217--234.

\bibitem{crvena} J. E. Pe\v{c}ari\'{c}, F. Proschan, Y. L. Tong, \textit{Convex functions, Partial orderings and statistical applications}, Academic Press Inc., San Diego 1992.

\bibitem{moh} M. Sababheh, \emph{Improved Jensen's inequaliy}, Math. Inequal. Appl., {\bf 20} (2), (2017), 389--403.

\bibitem{sila} Z. K. Silagadze, \textit{Citations and the Zipf–Mandelbrot Law} Complex Systems, 1997 (11), pp: 487–499.

\bibitem{zipf} G. K. Zipf, \textit{The Psychobiology of Language}, Cambridge, Houghton-Mifflin 1935.

\bibitem{zipf1} G. K. Zipf, \textit{Human Behavior and the Principle of Least Effort}, Reading,  Addison-Wesley 1949.

 

\end{thebibliography}
\end{document}